\documentclass{article}

\usepackage{maa-monthly}
\usepackage{asymptote}
\usepackage{gensymb}

\theoremstyle{theorem}
\newtheorem{theorem}{Theorem}
\newtheorem{lemma}{Lemma}
\newtheorem{corollary}{Corollary}

\theoremstyle{definition}

\newcommand{\ha}{\frac{1}{2} \degree}

\begin{document}

\title{Adventitious angles problem: the lonely fractional derived angle}
\markright{The lonely fractional derived angle}

\author{Yong Kong and Shaowei Zhang}

\maketitle

\begin{abstract}
In the ``classical'' adventitious angle problem,
for a given set of three angles $a$, $b$, and $c$ measured in integral degrees
in an isosceles triangle, 
a fourth angle $\theta$ (the derived angle), 
also measured in integral degrees, 
is sought.
We generalize the problem to find $\theta$ in fractional degrees.
We show that the triplet $(a, b, c)=(45\degree, 45\degree, 15\degree)$
is the only combination that leads to  
 $\theta = 7\ha$ as the fractional derived angle.
\end{abstract}

%
%
%

\noindent

The familiar problem of adventitious angles 
deals with an isosceles triangle
as illustrated in Figure.~\ref{F:tri}
where $AB=AC$.
The problem asks to find the angle $\theta$ (the \emph{derived} angle)
for certain given values of the angles $a$, $b$, and $c$.
In the following we will use the notation
$(a, b, c; \theta)$ to denote a problem and its solution.

The problem is also known as Langley's problem,
who first proposed the particular case of 
$(20\degree, 60\degree, 50\degree; 30\degree)$
\cite{langley1922}.
The angles are called ``adventitious'' due to the
fact that among $113564$ possible triplets of 
$a$, $b$, and $c$ 
(with disregard for mirror images and trivial cases of $b=c$),
only $53$ of them yield an angle of $\theta$ with integral numbers of degrees
\cite{tripp1975}.

These seemingly innocent problems
are more difficult than their simple appearances
\cite{coxeter1967}. 
Straightforward use of angle chasing
will not be sufficient to solve these problems.
Ingenious constructions are usually needed
\cite{quadling1977, quadling1978, sakashita1977, tripp1975}.

\begin{figure}[h]
\begin{center}


\includegraphics[angle=0,width=0.6\columnwidth]{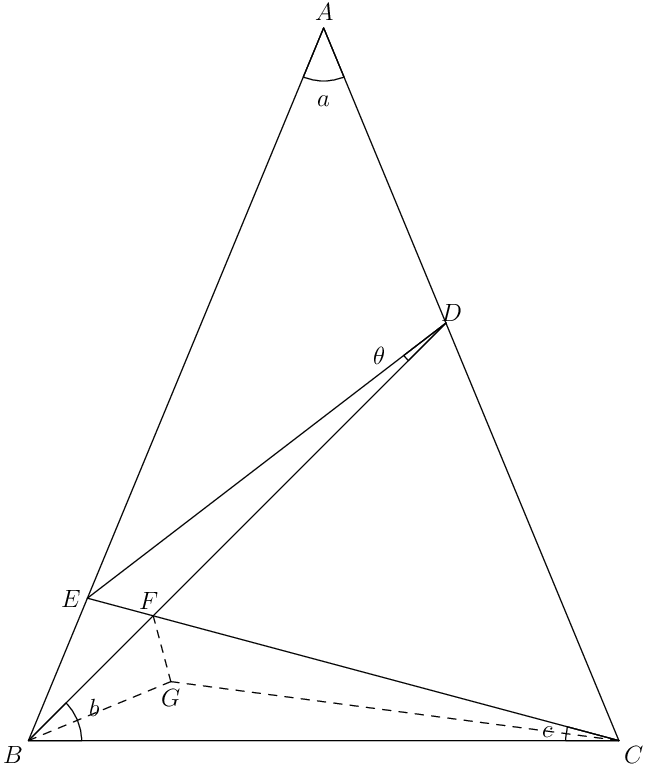}

\end{center}
\caption{
\label{F:tri}
}
\end{figure}

In the ``classical'' adventitious angle problem,
the given angles $a$, $b$, and $c$
as well as the derived angle $\theta$
are all measured in integral degrees.
We generalize the problem to find $\theta$ in fractional degrees
while still keeping the angles $a$, $b$, and $c$
in integral degrees.
We'll first prove that for the angle $\theta$ to be fractional degrees,
its denominator is at most $2$.
A computer search shows that 
$(45\degree, 45\degree, 15\degree; 7\ha)$
is the only potential fractional solution.
For this particular case 
we'll give a pure geometrical proof and a trigonometric proof
that 
$(45\degree, 45\degree, 15\degree)$ is the only adventitious set
of angles that has fractional derived angle
$\theta=7\ha$.

\begin{theorem}
\label{T:1}
For the angle $\theta$ to be a rational number of degrees
when angles $a$, $b$, and $c$ are measured in integral degrees, 
the denominator of $\theta$
is at most 2.
\end{theorem}

To prove the theorem we first prove the following three lemmas. 
In the following, $\mathbb{Q}$ and $\mathbb{R}$ stand for the fields of rational and real numbers.
Let $n$ be a positive integer and $\alpha = \frac{2\pi}{n}$.
Let $i=\sqrt{-1}$ and
denote  $\zeta_n$
an $n$th root of unity
$e^{\frac{2\pi i}{n}}
= e^{\alpha i}$,
$\mathbb{Q}_n = \mathbb{Q}(\zeta_n)$
the $n$th cyclotomic field,
and $\mathbb{Q}_n^+$
the maximal real subfield of $\mathbb{Q}_n$.

\begin{lemma} \label{L:1}
If $4|n$, 
then $\mathbb{Q}_n^+ = \mathbb{Q}_n \cap \mathbb{R}
= \mathbb{Q} \left( \tan(\frac{\alpha}{2}) \right)$.
\end{lemma}
\begin{proof}
By definition $\mathbb{Q}_n^+ = \mathbb{Q} \left( \zeta_n + \zeta_n^{-1} \right)
= \mathbb{Q} \left( \cos \alpha  \right)$.
Since $4|n$, we have  $i \in \mathbb{Q}_n$. 
Hence $\sin \alpha = (\zeta_n - \zeta_n^{-1}) / (2i) \in \mathbb{Q}_n^+$.
The identity 
$
 \tan (\frac{\alpha}{2}) = \frac{1-\cos \alpha}{\sin \alpha}
$
shows $\tan (\frac{\alpha}{2}) \in \mathbb{Q}_n^+$,
so $\mathbb{Q} \left( \tan(\frac{\alpha}{2}) \right) \subset \mathbb{Q}_n^+$.
On the other hand,
$
 \cos \alpha=
\frac{1-\tan^2(\frac{\alpha}{2})}{1+\tan^2(\frac{\alpha}{2})},
$
so
$\mathbb{Q}_n^+
= \mathbb{Q} \left( \cos \alpha  \right)
\subset
\mathbb{Q} \left( \tan (\frac{\alpha}{2}) \right)$.
\end{proof}

\begin{lemma} \label{L:2}
Assume $4|n$ and $4|m$. If $\mathbb{Q}_m^+ \subset \mathbb{Q}_n^+$, 
then $m|n$.
\end{lemma}

\begin{proof}
From the assumptions we have
$\mathbb{Q}_m = \mathbb{Q}_m^+(i) \subset 
\mathbb{Q}_n^+(i) = \mathbb{Q}_n$.
When $n$ is even, the only roots of unity in $\mathbb{Q}_n$
are the $n$th roots of unity
\cite[Corollary 3, p. 19]{marcus1977}.
Since $\zeta_m \in \mathbb{Q}_m \subset \mathbb{Q}_n$,
$\zeta_m$ must be one of the $n$th roots of unity: 
$\zeta_m^n = e^{\frac{2 \pi i n}{m}} = 1$.
This happens if and only if $m|n$.
\end{proof}

\begin{lemma} \label{L:3}
If $\beta=\frac{2\pi}{360m}$ for some integer $m$ 
such that $\tan \beta \in \mathbb{Q}_{360}^+$, then $m|2$.
\end{lemma}

\begin{proof}
From Lemma~\ref{L:1}, $\mathbb{Q}_{180m}^+ = \mathbb{Q}(\tan \beta) 
\subset \mathbb{Q}_{360}^+
$,
hence $180m|360$ by Lemma \ref{L:2}, so $m|2$.
\end{proof}

Proof of Theorem~\ref{T:1}:
\begin{proof}[\unskip\nopunct]
A formula for $\tan \theta$ can be obtained by applications of sine rule
\cite{tripp1975}
\[
 \tan \theta = 
\frac{ \sin(b+c) \sin (c) (\cos a + \cos 2b )}
     { \sin (b) (\cos a + \cos 2c) + \cos(b+c) \sin (c) (\cos a + \cos 2b) } .
\]
The conclusion follows from Lemma \ref{L:3}.
\end{proof}

\begin{theorem}
The only fractional derived angle, when angles $a$, $b$, and $c$
are integral numbers of degrees,
is $\theta=7\ha$ 
when 
the triplet $(a, b, c)=(45\degree, 45\degree, 15\degree)$.
\end{theorem}
\begin{proof}

A quick computer search shows that $(45\degree, 45\degree, 15\degree; 7\ha)$
is the only potential case.
The search is carried out using software package Maple (version 12) 
with $100$ decimal digits.
We will give two proofs that $(45\degree, 45\degree, 15\degree; 7\ha)$
is a solution:
an elementary geometry proof and a trigonometric proof.

For the geometrical proof,
let $G$ be the incenter of triangle
$\triangle BFC$ where three angle bisectors of the triangle meet
(Figure~\ref{F:tri}). 
Since 
$\angle{EBF} = \angle{GBF} = 22\ha$
and $\angle{BFE} = \angle{BFG} = 60\degree$,
it follows that
$\triangle{BEF} \cong \triangle{BGF}$.
Hence $BE=BG$.
Since $\angle{BDC} = \angle{DCB} = 67\ha$,
we have $BC=BD$,
which leads to
$\triangle{BED} \cong \triangle{BGC}$.
Hence 
$\theta = \angle{EDB} = \angle{GCB} = 7\ha$.

For the trigonometric proof, we use the identity from 
Eq. (6) of \cite{quadling1977}:
\[
 \frac{\sin \theta}{ \sin(b+c-\theta) }
= \frac{\cos(b + \frac{a}{2}) \sin (c)  \cos(b-\frac{a}{2}) }
        {\cos(c - \frac{a}{2}) \sin (b)  \cos(c+\frac{a}{2})}.
\]
Substituting the values of $a$, $b$, and $c$ 
into both sides of this identity, 
and simplifying the right side,
we have
\begin{align*}
 \frac{\sin \theta}{ \sin(\frac{\pi}{4} + \frac{\pi}{12}-\theta) }
&= \frac{\cos \frac{3\pi}{8} \sin \frac{\pi}{12} \cos \frac{\pi}{8} }
        {\cos \frac{\pi}{24} \sin \frac{\pi}{4}  \cos \frac{5\pi}{24}}\\
&= \frac{(2 \sin \frac{\pi}{8} \cos \frac{\pi}{8} )\sin \frac{\pi}{24} \sin \frac{\pi}{12} }
        {(2 \sin \frac{\pi}{24} \cos \frac{\pi}{24})  \sin \frac{\pi}{4}  \cos \frac{5\pi}{24} } \\
&= \frac{\sin \frac{\pi}{24}}
        {\sin (\frac{\pi}{4} + \frac{\pi}{12} - \frac{\pi}{24})}.
\end{align*}
Hence $\theta=\frac{\pi}{24} = 7\ha$.
\end{proof}

Theorem~\ref{T:1} can be generalized to basic units other than degree
($\pi/180$) \cite{tripp1975}.

\begin{corollary}
Given the set of angles $a$, $b$, and $c$ each in multiples of $\pi/N$ radians where $N$ is a positive integer divisible by $4$, 
if $\theta = k \pi/N$ and $k$ is a rational number,
then the denominator of $k$ is at most $2$.
\end{corollary}


\begin{biog}

\item[Yong Kong] (yong.kong@yale.edu)
\begin{affil}
Department of Molecular
Biophysics and Biochemistry\\
W.M. Keck Foundation Biotechnology Resource Laboratory \\
Yale University\\
333 Cedar Street, New Haven, CT 06510
\end{affil}
\item[Shaowei Zhang] (shaowei.zhang@yale.edu)
\begin{affil}
ITS\\
Yale University\\
25 Science Park, New Haven, CT 06511
\end{affil}
\end{biog}

\vfill\eject

\end{document}